\documentclass[11pt,reqno,oneside]{amsart}

\usepackage{amsmath}
\usepackage{array}
\usepackage{mathabx}
\usepackage{enumerate}
\usepackage{mathrsfs}
\usepackage{dsfont}
\usepackage{amssymb}
\usepackage{graphics}
\usepackage{shuffle}
\usepackage{array}
\newcolumntype{M}[1]{>{\centering\arraybackslash}m{#1}}

\input{amssym.def}
\input{amssym}

\makeatletter
\def\section{\@startsection{section}{1}%
  \z@{1.2\linespacing\@plus\linespacing}{.9\linespacing}%
  {\normalfont\bfseries\large}}
\makeatother

\newtheorem{theorem}{Theorem}

\newtheorem{proposition}{Proposition}
[section]
\newtheorem{lemma}[proposition]{Lemma}

\theoremstyle{definition}

\newtheorem{remark}[proposition]{Remark}

\makeatletter

\@addtoreset{equation}{section}
\makeatother

\newcounter{parage}[section]

\newcounter{parag}[subsection]

\newcounter{paraga}[subsection]

\newcommand{\cL}{{\mathcal L}}
\newcommand{\cN}{{\mathcal N}}

\newcommand{\cH}{\mathcal{H}}
\newcommand{\cR}{{\mathcal R}}

\newcommand{\be}{\begin{equation}}
\newcommand{\ee}{\end{equation}}
\newcommand{\bea}{\begin{eqnarray}}
\newcommand{\eea}{\end{eqnarray}}

\newcommand{\ad}{\operatorname{ad}}

\newcommand{\N}{\mathbb N}

\newcommand{\R}{\mathbb R}

\newcommand{\Z}{\mathbb Z}
\newcommand{\C}{\mathbb C}
\newcommand{\Q}{\mathbb Q}

\newcommand{\kk}{\mathbf{k}}
\newcommand{\KK}{\mathbf{K}}

\newcommand{\De}{\Delta}
\newcommand{\de}{\delta}
\newcommand{\lam}{\lambda}

\newcommand{\ph}{\varphi}

\newcommand{\un}{{\underline n}}
\newcommand{\ula}{{\underline\lam}}
\newcommand{\rula}{{r({\underline\lam})}}
\newcommand{\ua}{{\underline a}}
\newcommand{\ub}{{\underline b}}

\newcommand{\UN}{{\underline\cN}}

\newcommand{\est}{{\scriptstyle\diameter}}

\newcommand{\na}{\nabla}

\newcommand{\I}{{i}}  

\newcommand{\ex}{{\mathrm e}}

\newcommand{\ID}{\mathop{\hbox{{\rm Id}}}\nolimits}

\newcommand{\ti}{\tilde}

\newcommand{\sh}[3]{\operatorname{sh}\!\big( \begin{smallmatrix}#1,\,
#2\\#3\end{smallmatrix} \big)}
\newcommand{\shabn}{\sh{\ua}{\ub}{\un}}
\newcommand{\shabl}{\sh{\ua}{\ub}{\ula}}

\newcommand{\ens}{\enspace}
\newcommand{\col}{\colon\thinspace}               
\makeatletter
\newcommand*{\defeq}{\mathrel{\rlap{%
                     \raisebox{0.3ex}{$\m@th\cdot$}}%
                     \raisebox{-0.3ex}{$\m@th\cdot$}}%
                     =}
\makeatother                                                   

\newcommand{\ie}{{{i.e.}}\ }
\newcommand{\eg}{{{e.g.}}\ }

\newcommand{\lhs}{{left-hand side}}

\newcommand{\End}{\operatorname{End}}
\newcommand{\ord}{\operatorname{ord}}

\newcommand{\ii}{^{-1}}

\newcommand{\wt}{\widetilde}

\newcommand{\eps}{\varepsilon}

\newcommand{\bp}{U_+}
\newcommand{\tbp}{\wt U_+}
\newcommand{\bm}{U_-}
\newcommand{\tbm}{\wt U_-}
\newcommand{\idmoul}{{\mathds 1}}

\newcommand{\norm}[1]{\lVert #1 \rVert}

\newcommand{\dir}[3]{\langle #1 \,|#2|\, #3 \rangle}
\newcommand{\adir}[2]{|#1\rangle\,\langle #2| }

\makeatletter
\DeclareRobustCommand{\vf}{\@ifstar\@vf\@@vf}
\newcommand{\@vf}[2]{[ #1, #2 ]_{\mathrm{vf}}}
\newcommand{\@@vf}[2]{\left[ #1, #2 \right]_{\mathrm{vf}}}
\DeclareRobustCommand{\ham}{\@ifstar\@ham\@@ham}
\newcommand{\@ham}[2]{[ #1, #2 ]_{\mathrm{ham}}}
\newcommand{\@@ham}[2]{\left[ #1, #2 \right]_{\mathrm{ham}}}
\DeclareRobustCommand{\qu}{\@ifstar\@qu\@@qu}
\newcommand{\@qu}[2]{[ #1, #2 ]_{\mathrm{qu}}}
\newcommand{\@@qu}[2]{\left[ #1, #2 \right]_{\mathrm{qu}}}
\makeatother

\newcommand{\bet}{\beta}

\newcommand{\gB}{{\mathscr B}}

\newcommand{\gE}{{\mathscr E}}
\newcommand{\gL}{{\mathscr L}}
\newcommand{\gR}{{\mathscr R}}

\newcommand{\gU}{{\mathscr U}}

\newcommand{\imp}{\quad \Rightarrow \quad}

\newcommand{\LIE}{\operatorname{Lie}}
\newcommand{\Span}{\operatorname{Span}}

\newcommand{\Bun}{B_{[\,\un\,]}}

\newcommand{\Best}{B_{[\est]}}
\newcommand{\gBun}{\gB_{[\,\un\,]}}

\newcommand{\kUN}{\kk\,\UN}
\newcommand{\KUN}{\KK\,\UN}

\newcommand{\ebas}{\mathbf{e}}

\newcommand{\LCE}{\cL^\C_\ebas}
\newcommand{\LRE}{\cL^\R_\ebas}

\newcommand{\hb}{{\mathchar'26\mkern-9mu h}}
\renewcommand{\hbar}{\hb}

\newcommand{\ugeq}[1]{_{\ge #1}}

\usepackage[usenames,dvipsnames]{color}
\usepackage[hyperindex=true, colorlinks=false]{hyperref}

\newcommand{\bcf }
{\begin{color}{black} }

\newcommand{\ecf }
{\end{color} }

\newcommand{\ff}{}
\newcommand{\fff}
[1]{\textcolor{black}{#1}}
\newcommand{\ffff}
[1]{\textcolor{black}{#1}}

\begin{document}

\title[]{
Rayleigh-Schr\"odinger series and Birkhoff decomposition}

\author{Jean-Christophe Novelli}
\address{\scriptsize Universit\'e Paris-Est Marne-la-Vall\'ee, Laboratoire d'Informatique
Gaspard-Monge (CNRS - UMR 8049), 77454 \ \ \ Marne-la-Vall\'ee cedex 2, France\ (Novelli and Thibon)}
\email{\scriptsize novelli@univmlv.fr,\  jyt@univmlv.fr}

\author{Thierry Paul}
\address{\scriptsize 
CMLS, Ecole polytechnique, CNRS, Universit\'e Paris-Saclay, 91128 Palaiseau Cedex, France
}
\email{\scriptsize thierry.paul@polytechnique.edu}


\author{David Sauzin}
\address{\scriptsize {CNRS UMR 8028 -- IMCCE}\\
{Observatoire de Paris}\\
{77 av. Denfert-Rochereau}\\
{75014 Paris, France}}
\email{\scriptsize david.sauzin@obspm.fr}

\author{Jean-Yves Thibon}

\date{}
\maketitle


\begin{abstract}
We derive new expressions for the Rayleigh-Schr\"odinger series
describing the perturbation of eigenvalues of quantum
Hamiltonians. The method, somehow close to the so-called dimensional
renormalization in quantum field theory, involves the Birkhoff
decomposition of some Laurent series built up out of explicit fully
non-resonant terms present in the usual expression of the
Rayleigh-Schr\"odinger series. 
Our results provide
  new combinational formulae and a new way \ff{of  deriving} perturbation series in Quantum Mechanics. More generally we prove that such a decomposition provides solutions 
of 
  general normal form problems in Lie algebras.
\end{abstract}

\section{Introduction}\label{intro}
Rayleigh-Schr\"odinger expansion is a powerful tool in quantum mechanics, chemistry and more generally applied sciences. It consists in expanding the spectrum of an operator (finite or infinite dimensional) which is a perturbation of a bare one, around the unperturbed spectrum. Besides, let us mention that perturbation theory has been a clue in the discovery   of quantum dynamics by Heisenberg in 1925 \cite{B,H}.
Considering the \ff{huge} bibliography on the subject, we only quote in the present article  the two classical textbooks \cite{kato,reedsimon}, and present in this introduction an elementary formal derivation of the Rayleigh-Schr\"odinger expansion.

Let us consider a self-adjoint operator~$H_0$ on a Hilbert space~$\cH$ whose spectrum $\{E_0(n),\ n\in J\subseteq\mathbb N\}$ is supposed (for the moment) to be
discrete and non-degenerate, and   a perturbation $V$ of $H_0$, namely
a self-adjoint bounded operator of \fff{``small size"}. It is well-known
that one can  unitarily conjugate $H\defeq H_0+V$, formally at any order in the size of $V$, to an operator of the form $H_0+N$ where $N$ is diagonal on the eigenbasis of $H_0$.
More precisely
\be\label{poi}
\exists C\mbox{, unitary, such that }C(H_0+V)C\ii \sim H_0+N,\ \ \  [H_0,N]=0,
\ee
the symbol~$\sim$ meaning (in the good cases) that $\norm{C(H_0+V)C^{-1}- (H_0+N)}=O(\norm{V}^\infty)$, for some suitable norm $\norm{\cdot}$.

An elegant way of building this pair $(N,C)$ consists in using the
so-called Lie algorithm, see e.g. \cite{DGH}: let us look for~$C$ of
the form  $C=\ex^{\frac{1}{i\hbar}W}$ with~$W$ self-adjoint (which
will ensure that~$C$ is unitary). 
Expanding  $W = W_1+W_2+\cdots$ and $N=N_1+N_2+\cdots$ ``in powers of
$V$", and using \fff{Hadamard's lemma}
$\ex^{\frac{1}{i\hbar}W} H \ex^{-\frac{1}{i\hbar}W} = H + \sum\limits_{k=1}^\infty
\frac1{k!}\big(\frac{1}{i\hbar}\big)^k
[\, \underset{\fff{\text{$k$ times}}}{\underbrace{W,[W,\dots [W}} , H]\dots]]$, 
we get
\begin{align}
\frac{1}{i\hbar} [H_0,W_1] + N_1 &=\fff{ V_1,} \quad
\fff{V_1 \defeq V}
 \nonumber\\[1ex]
\frac{1}{i\hbar}[H_0,W_2] + N_2 &= V_2, \quad
V_2 \defeq \frac{1}{i\hbar}[W_1,V]-\frac1{2\hbar^2}[W_1,W_1,H_0]] \nonumber\\[1ex]
\frac{1}{i\hbar}[H_0,W_3] + N_3 &= V_3, \quad
V_3 \defeq \frac{1}{i\hbar}[W_2,V]
-\frac 1{2\hbar^2}[W_1,W_1,V]]
-\frac1{2\hbar^2}[W_1,[W_2,H_0]] \nonumber\\
& \hspace{13em} -\frac1{2\hbar^2}[W_2,[W_1,H_0]]
                          -\frac{1}{6i\hbar^3}[W_1,[W_1,[W_1,H_0]]] \nonumber\\
&\vdots \label{hier}\\
\frac{1}{i\hbar}[H_0,W_k] + N_k &= V_k, \quad
V_k \defeq \dots \nonumber
\end{align}
These equations, together with the commutation relations
$[H_0,N_k]=0$, are solved recursively by
\begin{align*}
\big(e_n,N_ke_m\big) &= \big(e_n,V_ke_{n}\big) \de_{n m}, \\[1ex]
\big(e_n,W_ke_{m}\big)&= \frac {i\hbar}{E_0(n)-E_0(m)}
\big(e_n,V_ke_{m}\big) \ens\text{if $n\neq m$,}
\end{align*}
where we have denoted by~$e_n$ an eigenvector of~$H_0$ of eigenvalue~$E_0(n)$,
because $H_0$ has simple spectrum and
$\big( e_n, \frac{1}{i\hbar}[H_0,A]e_m \big) = 
\frac{1}{i\hb}(E_0(n)-E_0(m)) \big(e_n,A e_m\big)$ for an arbitrary operator~$A$.
Note that $N_k = \mbox{Diag}(V_k)$ (diagonal part of~$V_k$ on the
eigenbasis of $H_0$), but the diagonal part of $W_k$ remains
undetermined;
one can check that the $N_k$'s are uniquely determined by~\eqref{poi}.

Using the Dirac notation
$\dir{n}{A}{m} \defeq \big(e_n,A e_{m}\big)$ for an arbitrary operator~$A$,
we easily arrive at
\be\label{rstand}
\dir{n}{N_k}{n}=\sum_{n_1,n_2,\dots,n_{k-1}}c_{n_1,\dots,n_{k-1},n}\dir{n}{V}{n_1}\dir{n_1}{V}{n_2}\cdots\dir{n_{k-1}}{V}{n}
\ee
where the coefficients $c_{n_1,\dots,n_{k-1},n}$ have to be determined recursively.

This is the standard way the Rayleigh-Schr\"odinger series is usually expressed: the correction to the eigenvalue $E_0(n)$ is given at order $k$ by the r.h.s.\ of~\eqref{rstand}, that is the diagonal matrix elements of the (diagonal) normal form $N_k$.

\medskip


Looking at the hierarchy of equations~\eqref{hier}, one realises that only commutators should be involved in~\eqref{rstand} for $k\geq 2$. One way of  achieving  this has been developed recently by two of us in \cite{Dyn}:
let
\be    \label{eqdefBlamQU}
\cN \defeq \big\{\, \tfrac{1}{\I\hb}(E_0(\ell)-E_0(k)) \mid 
k,\ell\in\mathbb N\,\big\},
\ee
and define, for $\lambda\in\cN$, 
\be\label{homo}
V_\lambda \defeq
\sum_{ \substack{(k,\ell) \,\text{such that}\\[.5ex]
E_0(\ell)-E_0(k)=\I\hb\lam} }\dir{k}{V}{\ell}\,\adir{\ell}{k}
\ee
with the Dirac notation $\adir{\ell}{k}\psi \defeq (e_\ell,\psi)e_k$
for an arbitrary vector~$\psi$,
so that 
\be\label{finit}
\frac1{i\hbar}[H_0,V_\lambda]=\lambda V_\lambda
\ens\text{and}\ens
V=\sum\limits_{\lambda\in\cN}V_\lambda.
\ee
We will suppose that $V$ is finite-band, that is to say that the sum in~\eqref{finit} is finite.
\fff{According to \cite{Dyn},}
\emph{for every $k\geq 1$
there exist coefficients $R^{\lambda_1,\dots,\lambda_k}\fff{\in\C}$ such that}
\be\label{uniq}
N_k=\sum_{\lambda_1,\dots,\lambda_k\in\cN}\tfrac1k{R^{\lambda_1,\dots,\lambda_k}}
\tfrac1{i\hbar}[V_{\lambda_1},\tfrac1{i\hbar}[V_{\lambda_2},\dots\tfrac1{i\hbar}[V_{\lambda_{k-1}},V_{\lambda_k}]\dots]].
\ee
\noindent The coefficients $R^{\lambda_1,\dots,\lambda_k}$ are
computable recursively together with coefficients
$S^{\lambda_1,\dots,\lambda_k}$ appearing in a  similar expansion for
the formal unitary operator~$C$ --- see~\eqref{eqmouldexpCS}.

\medskip


The family of pairs
$(R^{\lambda_1,\dots,\lambda_k},S^{\lambda_1,\dots,\lambda_k})$
is obtained by solving a universal ``mould equation" (independent of $V$ and
depending on $H_0$ only through $\cN$) studied in \cite{Dyn} and
recalled in the next section. 
In general, this mould equation has more than one solution
(the set of all solutions is described in \cite{Dyn}), so the
decomposition~\eqref{uniq} is not unique, though $N_k$ is.
Using~\eqref{homo} and introducing a decomposition of the identity on
the unperturbed eigenbasis in~\eqref{uniq}, one would certainly
recover~\eqref{rstand}, but probably with a big combinatorial
complexity in the expressions as $k\to\infty$.

One of the main goal of this note is to introduce a new (to our
knowledge) way of finding a solution to the mould equation, and thus a
family of coefficients $R^{\lambda_1,\dots,\lambda_k}$ satisfying~\eqref{uniq}. 
It consists in applying a method actually very similar to the so-called
dimensional regularisation in quantum field theory (but much simpler):
we will add a dependence in an undetermined parameter~$\eps$. 
This will lead us to a modified mould equation with a unique
solution, for which the coefficients are given by explicit Laurent
series
\[
T^{\lam_1,\ldots,\lam_k}(\eps) \in K \defeq \C((\eps)).
\]
The correct expression for $R^{\lambda_1,\dots,\lambda_k}$ will then
be obtained, up to a factor~$k$, by taking the residue of the polar
part of the so-called ``Birkhoff decomposition'' of this family of Laurent
series, relative to the decomposition
\[
K=K_+\oplus K_-, \qquad
K_+\defeq \C[[\eps]], \qquad K_-=\eps^{-1}\C[\eps^{-1}].
\]

More precisely, let $\underline\cN$ be the set of words on the alphabet $\cN$
  (finite sequences of elements of $\cN$)
  and denote by $r(\underline\lambda)$ the length of the word
  $\underline\lambda=\lambda_1\lambda_2\cdots \lambda_{r(\ula)}$.
  We will consider the set of functions from $\underline\cN$ to $K$,
  that is 
\be \label{eqdefKUN}
K^\UN \defeq \{M \col \ula\in\UN \mapsto M^\ula \in K \}.
\ee
On~$\UN$, we define word concatenation by
$\underline\lambda\,\underline\lambda'=\lambda_1\cdots\lambda_\ell\lambda'_1\cdots\lambda'_m$
for $\underline\lambda=\lambda_1\cdots\lambda_\ell,\
\underline\lambda'=\lambda'_1\cdots\lambda'_m$
and, on $K^{\underline\cN}$, we define the product
\be \label{eqdefmouldprod}
(M\times N)^{\underline{\lambda}} \defeq
\sum_{\underline{a}\,\underline{b}=\underline{\lambda}}M^{\underline{a}}N^{\underline{b}}\in K
\ee
with unit $\idmoul\in K^{\underline N}$ defined by $\idmoul^\est=1_K$
and $\idmoul^{\ula}=0$ for $\ula\neq\est$.

Let $T:\ \underline\cN\to K,\ \underline\lambda\mapsto T^{\underline{\lambda}}(\eps)$ be given by
\[
T^{\underline{\lambda}}(\eps)\defeq \frac1{(\lambda_1+\eps)(\lambda_1+\lambda_2+2\eps)\cdots(\lambda_1+\cdots+\lambda_{r(\underline{\lambda})}+r(\underline{\lambda})\eps)},
\]
considered of course as a formal Laurent series in $\eps$
(note that $T^{\underline{\lambda}}(\eps)\in K_+$ only for those
$\ula\in\UN$ such that the partial sums $\lam_1+\cdots+\lam_j$ are all
nonzero).
Its ``Birkhoff decomposition" is the following:
%
\bcf
\emph{there exists a unique pair $(\bp,\bm)\in K^\UN\times K^\UN$ such that}
\[
\bm^\est=\bp^\est=1_K, \qquad 
\bm-\idmoul \in K_-^{\underline\cN},\qquad 
\bp\in K_+^{\underline\cN},\qquad
\bm\times
T=
 \bp.
\]
This will be proved in the next section as Proposition~\ref{decomp}
(in a more general setting), 
\ecf
together with recurrence relations in order to compute $\bm$ and $\bp$.

Since $\bm\in \idmoul+K_-^{\underline\cN}$, one can evaluate  $\eps \bm^{\underline{\lambda}}(\eps)$ at $\eps=\infty$ for each word $\underline\lambda\neq \est$. We are now in position of stating one of the main results of this article.

\begin{theorem}\label{main}
For any $k\geq 1$, one can write
\be\label{rsnew} 
N_k=\sum\limits_{\lambda_1,\dots,\lambda_k\in\cN}
N^{\lambda_1\cdots\lambda_k}\ 
\tfrac1{i\hbar}[V_{\lambda_1},\tfrac1{i\hbar}[V_{\lambda_2},\dots\tfrac1{i\hbar}[V_{\lambda_{k-1}},V_{\lambda_k}]\dots]],
\ee 
with
\be
N^{\lambda_1\cdots\lambda_k} \defeq 
-\,\text{residue of }\bm^{\lambda_1\cdots\lambda_k}
=-[\eps \bm^{\lambda_1\cdots\lambda_k}(\eps)]_{\eps=\infty}.
\ee
\end{theorem}
We will prove much more in the following sections. In particular we
will show that the coefficients $S^\ula \defeq \bp^{\underline{\lambda}}(\eps)|_{\eps=0}$
give rise to a formal unitary operator
\be   \label{eqmouldexpCS}
C = \sum\limits_{k=0}^\infty \; \sum\limits_{\lambda_1,\dots,\lambda_k\in\cN}\, 
(\tfrac1{i\hbar})^k S^{\lambda_1\cdots\lambda_k}
V_{\lambda_1} V_{\lambda_2} \cdots V_{\lambda_{k}}
\ee
which satisfies the conjugacy equation~\eqref{poi}.
We will also remove the simplicity condition on the spectrum of~$H_0$.

\medskip


This paper is organized as follows. In Section~\ref{sect2} we briefly
recall elements of \fff{Ecalle's mould calculus} 
\fff{(\ie the manipulation of families of coefficients indexed by words)}
and, \fff{in the more general setting of
a normalization problem in a complete filtered Lie algebra~$\cL$}, the
mould equation implying~\eqref{poi}; then we prove the underlying
Birkhoff decomposition and the main \fff{results} of this article,
\fff{Theorems~\ref{mainmain} and~\ref{corolBD}}. In Section~\ref{mainmainq} we prove the
general ``quantum" result, \fff{Theorem~\ref{mainq},} implying
Theorem~\ref{main}. In Section~\ref{sect3} we present different
situations where Theorem~\ref{corolBD} applies, including
perturbations of Hamiltonian vector fields in classical
dynamics. 
\fff{For the sake of completeness, we have included \ffff{the
    derivation of} the mould equation \ffff{in appendix}.}
%

\medskip


The techniques used here have been introduced in various papers
dealing with normal forms of dynamical systems \fff{\cite{EV,men,men2,Dyn,
  Norm}}, 
quantum mechanics \cite{Dyn, Norm} 
and renormalization in QFT \cite{CK}. Some of them use the language of
Ecalle's mould calculus, while others \fff{rely only on the formalism
  of Hopf algebras}. 
%
%
\fff{%
The idea of using Birkhoff decomposition for normal form problems
appears in the pioneering work of F.~Menous\footnote{Since the first version of this
  paper has been posted, we have learnt that F. Menous
  \cite{M-postdam} had announced  results in the same line of research.} \cite{men,men2}.
The article \cite{men2} notably deals with an abstract Lie-algebraic
context, however it considers completed graded Lie algebras with
finite-dimensional components and does not express the results in
terms of mould expansions, whereas, for our most general result
\ffff{and its application to the Rayleigh-Schr\"odinger expansion}, we
need complete filtered Lie algebras without dimensional restriction,
and we aim at emphasizing the explicit character of the
coefficients which are involved in the solution of the normalization
problem 
(correspondingly, we apply the Birkhoff decomposition to an element of the mould
algebra, rather than to an element of the enveloping algebra of~$\cL$).%
}
The algebraic expansion that we obtain for the \lhs\ of~\eqref{rsnew}
in Theorem~\ref{main} (or~\eqref{eqdefNihb} in Theorem~\ref{mainq})
is, to our knowledge, new. 
\ffff{We point out that no prerequisite on mould calculus or Hopf
  algebras is needed to read this article.}

%

\section{Mould calculus and Birkhoff decomposition}\label{sect2}
In full generality, we are interested in the following situation:
given \fff{$X_0\in \cL$ and $B \in \cL\ugeq1$},
where 
\[
\cL = \cL\ugeq 0 \supset\cL\ugeq1\supset\cL\ugeq2 \supset \ldots \qquad
\]
is a complete filtered Lie algebra\footnote{This means that
$[\cL\ugeq m,\cL\ugeq n] \subset \cL\ugeq{m+n}$ for all $m,n\in\N$,
%
%
$\bigcap\cL\ugeq m=\{0\}$ 
and~$\cL$ is a complete metric space for the distance
$d(X,Y) \defeq 2^{-\ord(Y-X)}$,
where we denote by $\ord \col \cL\to\N\cup\{\infty\}$ the order
function associated with the filtration
(function characterized by 
$\ord(X)\geq m \Leftrightarrow X\in\cL\ugeq m$).} 
over a field~$\kk$ of characteristic zero,
%
%
we look for a Lie algebra automorphism~$\Psi$ which maps $X_0+B$ to an
element $X_0+N$ of~$\cL$ which commutes with~$X_0$:
\be\label{normal}
\Psi(X_0+B) = X_0+N,\qquad [X_0,N]=0,\qquad 
\Psi\in\mbox{Aut}(\cL).
\ee
Then~$\Psi$ is called a ``normalizing automorphism'' and $\fff{X_0+N}$ a ``normal form'' of $X_0+B$.
Our key assumption will be that~$B$ can be decomposed into a formally
convergent series $B =
\sum\limits_{n\in\cN} B_n$ of eigenvectors of the inner derivation
%
%
$\ad_{X_0} \col Y \mapsto [X_0,Y]$, namely 
\bcf \be \label{eqBneigenv}
[X_0,B_n]=\ph(n)B_n, \qquad
B_n \in \cL\ugeq1
\ens \text{for each $n\in\cN$,}
\ee \ecf
for some function $\ph \col \cN\to\kk$.

\bcf
In \cite{Dyn}, solutions $(N,\Psi)$ are constructed by means of the ansatz
\be\label{ansatz}
\left\{
\begin{array}{ccl}
N&=&
\sum\limits_{r\geq 1} \;\sum\limits_{n_1,n_2,\ldots,n_r\in\cN}\,
\frac{1}{r} R^{
n_1\cdots n_r} 
[B_{n_1},[\ldots[B_{n_{r-1}},B_{r}]\ldots]]\\
\Psi&=&
\sum\limits_{r\geq 0} \;\sum\limits_{n_1,n_2,\ldots,n_r\in\cN}\,
S^{
n_1\cdots n_r} 
\ad_{B_{n_1}} \cdots \ad_{B_{n_r}}
\end{array}
\right.
\ee
where $( R^{n_1\cdots n_r} )$ and $( S^{n_1\cdots n_r} )$
are suitable families of coefficients.
A family of coefficients indexed by all the words $n_1\cdots
n_r$ is called a ``mould''.
It is shown in \cite{Dyn} that~\eqref{ansatz} yields a solution as
soon as the moulds $( R^{n_1\cdots n_r} )$ and $( S^{n_1\cdots n_r} )$
satisfy a certain ``mould equation'', equation~\eqref{mouldeq} below.
We give in Section~\ref{secMoulds} the basics of mould
calculus and state the mould equation; we then show in
Sections~\ref{secBirkhDec}--\ref{secMainResult} a new method to solve
the mould equation.
\ecf

%

\subsection{Moulds} \label{secMoulds} 

Mould calculus has been introduced and developed by Jean \'Ecalle
(\cite{Eca81}, \cite{Eca93}) in the 80-90's, \fff{initially in relation
  with the free Lie algebra of alien operators in resurgence theory,
  providing also} powerful tools
for handling problems in local dynamics, typically the normalization
of vector fields or diffeomorphisms at a fixed point \fff{\cite{EV}}.

Let $\cN$ be a nonempty set and $\kk$ a \fff{commutative} ring.
Similarly to~\eqref{eqdefKUN}, we consider the set $\kk^\UN$ of
all families of coefficients $M^\ula$ indexed by words $\ula \in \UN$.
A ``$\kk$-valued mould'' is an element of $\kk^\UN$.
Mould multiplication is defined by~\eqref{eqdefmouldprod}
\fff{and makes $\kk^\UN$ a $\kk$-algebra}.
%
%
%
\bcf
A mould can be identified with a linear form on $\kUN$, the linear
span of the words; the mould product can then be identified with the
convolution product of linear forms corresponding to the comultiplication 
%
$\un \mapsto \sum\limits_{\un = \ua \, \ub} \ua
\otimes\ub$. 

The ``shuffle algebra'' is $\kUN$ viewed as a Hopf algebra, with the
previous comultiplication
(with counit $\est\mapsto 1_\kk$ and $\un \mapsto 0$ for a nonempty word~$\un$),
the antipode map $n_1\cdots n_r \mapsto (-1)^r n_r\cdots n_1$, 
and the ``shuffle product''~$\shuffle$, which can be recursively defined by the formula
\be 
\lam \ua \shuffle \mu \ub = \lam (\ua \shuffle\mu\ub) +
\mu (\lam \ua \shuffle \ub)
\ens \text{where $\lam ,\mu $ are letters and $\ua ,\ub$ are words} 
\ee
(the unit being~$\est$),
giving rise to structure coefficients $\shabn$ known as ``shuffling coefficients'':
\be   \label{eqshufflecoeff}
\ua \shuffle\ub = \sum_{\un\in\UN} \shabn\, \un
\ee
(see \eg Section~2.2 of \cite{Dyn} for their definition in terms of
permutations of $r(\un)$ elements:
$\shabn$ is the number of ways~$\un$ can be obtained by
interdigitating the letters of~$\ua$ and~$\ub$ while preserving their
internal order in~$\ua$ or~$\ub$).

By duality, this leads to \'Ecalle's definition of 
symmetrality, which is fundamental.
A mould~$M$ is said to be ``symmetral'' if the corresponding linear
form is a character of the shuffle algebra \cite{men}, \ie 
$M^\est=1_\kk$ and
\be   \label{eqcharchufflalg}
M^{\ua\shuffle \ub}=M^\ua M^\ub
\quad\text{for all $\ua,\ub\in\UN$,}
\ee
which boils down to the condition
$\sum\limits_{\un\in\UN} \shabn M^\un = M^\ua M^\ub$
for any  nonempty words $\ua, \ub$ \cite{Eca81}.
Its multiplicative inverse~$M\ii$ then coincides with the
mould~$\ti M$ defined by
\be   \label{eqantipod}
\ti M^{n_1\cdots n_r} \defeq (-1)^r M^{n_r\cdots n_1}
\ee
%
(this is \ff{a manifestation} of the antipode 
of the shuffle algebra; see \eg Proposition 5.2 of \cite{mouldSN}).

\ecf


For us, symmetrality is useful because whenever a mould~$S$ is
symmetral, the operator~$\Psi$ to which it gives rise by mould
expansion as in the second part of~\eqref{ansatz} is a Lie algebra
automorphism, and its inverse~$\Psi\ii$ is given by the mould expansion
associated with $S\ii=\ti S$. 
\bcf
%
This is because the $\ad_{B_{n_i}}$'s are derivations of the Lie
algebra~$\cL$, hence the composite operators
$\gB_\un \defeq \ad_{B_{n_1}} \cdots \ad_{B_{n_r}}$ satisfy the
generalized Leibniz rule
\be   \label{eqgenLeibn}
\gB_\un [X,Y] = \sum_{\ua,\ub\in\UN} \, \shabn [\gB_\ua X,\gB_\ub Y].
\ee
\ecf


\ff{Let us define the mould} $I_\kk\in\kk^\UN$ by
$I^{\underline{n}}_\kk=\delta_{r(\underline{n}),1}1_\kk$ 
and \ff{the operator} $M \mapsto \na_\ph M$ of $\kk^\UN$ by
\be\label{defM}
\na_\ph M^\un \defeq (\ph(n_1)+\cdots+\ph(n_{r(\un)}) M^\un
\defeq \ph(\un) M^\un,
\ee
\fff{with the eigenvalue function $\ph\col\cN\to\kk$
  of~\eqref{eqBneigenv} thus} extended to
a monoid morphism $\ph\col \UN\to\kk$.
%
%
%
\fff{These are the ingredients of a ``mould equation'', whose
  solutions $(R,S)$ yield solutions $(\Psi,N)$ of the normalization
  problem~\eqref{normal}, as proved in Section~3.4 of \cite{Dyn}:}

\begin{proposition}[\cite{Dyn}]\label{dyndyn}
{When~$\kk$} is a field of characteristic~$0$, 
equation~\eqref{normal} is solved by  the ansatz~\eqref{ansatz} if the
pair of moulds $(R,S)$ satisfies the ``mould equation''
\be\label{mouldeq}
\left\{
\begin{array}{l}
\nabla_\ph S=S\times I_\kk-R\times S\\
\nabla_\ph R=0\\
\text{$S$ symmetral.}\\
\end{array}
\right.
\ee
\end{proposition}

For the sake of completeness and \fff{clarity}, \fff{we give the proof in appendix.}

\ff{All the solutions of the mould equation~\eqref{mouldeq} are constructed in \cite{Dyn}}
(this is the generalization of some of the statements of the preprint
\cite{EV}, which introduced the mould equation in the context of local
holomorphic vector fields and diffeomorphisms).
%

We now show an alternative method to obtain a particular solution $(R,S)$.
From now on, we suppose that $\kk$ is a field of characteristic~$0$.

\subsection{Birkhoff decomposition} \label{secBirkhDec}


%
%

\fff{We call ``resonant'' any word~$\un$ such that $\ph(\un)=0$.
In the case when the function $\ph \col \cN \to \kk$ is
such that~$\est$ is the only resonant word,} 
it is easy to check that there is a unique
solution to~\eqref{mouldeq}, given by $R=0$
and
$S^{n_1\cdots
  n_r}=\frac1{\ph(n_1)\ph(n_1n_2)\cdots\ph(n_1\cdots n_r)}$,
\fff{but in general the latter expression is ill-defined.}
We will extend the field~$\kk$ to the field~$\KK$ of formal Laurent series
with coefficients in~$\kk$ 
\fff{and replace~$\ph$ by a $\KK$-valued function for which there is
  no resonant word except the empty one}. The new mould
equation~\eqref{mouldeq} will therefore be easily solvable by an
expression similar to the one just mentioned. The original situation
will be recovered by taking some kind of residue of the Birkhoff
decomposition of this explicit solution.

The Birkhoff decomposition
has been originally introduced by G.~D.~Birkhoff for matrices of Laurent series. It has been extended by Connes and
Kreimer \cite{CK} to Hopf algebras of Feynman diagrams, and abstract versions for general Hopf algebras
appear in several papers \cite{EFGM,Manchon}.

Let $\KK \defeq \kk((\eps))$ and $\KK_+\defeq \kk[[\eps]]$, so that
$\KK=\KK_+\oplus \KK_-$ with $\KK_-=\eps^{-1}\kk[\eps^{-1}]$.
Note that $\kk\subset\KK$, by identifying elements of $\kk$ with
constant formal series, so
$\kk^{\underline{\cN}}\subset\KK^{\underline{\cN}}$.
\fff{Let us consider the function 
$\Phi\col n \in\cN \mapsto \ph(n)+\eps 1_\kk \in \KK$
and, correspondingly,} the operator $M\mapsto \na_\Phi M$ of $\KK^\UN$ defined by
\[
\na_\Phi M^{\underline{n}}(\eps)
=
(\ph(\underline{n})+r(\underline{n})\eps)M^{\underline{n}}(\eps).
\]
Since $\KK$ is a field and $\ph(\un)+r(\un)\eps \neq 0$ for
$\un\neq\est$ (even if $\ph(\un)=0$!),
the mould equation associated with $(\Phi,\KK)$ (in place of
$(\ph,\kk)$) has a unique solution, given by $R=0$ and
\be\label{TouS}
T^{n_1\cdots n_r}(\eps)=\frac1
{(\ph(n_1)+\eps)(\ph(n_1n_2)+2\eps)\cdots(\ph(n_1\cdots n_r)+r \eps)}.
\ee
\fff{The symmetrality of~$T$ is easily cheked \eg by induction on the
  sum of the lengths of~$\ua$ and~$\ub$ in~\eqref{eqcharchufflalg}.}
Of course $T^{n_1\cdots n_r}(\eps)$, considered as a rational
function, is singular at $\eps=0$ when some words $n_1\cdots n_\ell,\
\ell\leq r$, are \fff{$\ph$-resonant}. 
%

\fff{Any $\KK$-valued symmetral mould can be interpreted} as a character of
\fff{the shuffle algebra $\KUN$ and,}
therefore, admits a Birkhoff decomposition with respect to the
decomposition $\KK=\KK_+\oplus \KK_-$ (see \eg \cite{Manchon}).
For the sake of completeness, we state this \fff{as a proposition} which we will
prove from scratch in the context of moulds. 

\begin{proposition}
\label{decomp}
Suppose~$T$ is an arbitrary $\KK$-valued symmetral mould.
Then there exists a unique pair $(\bp,\bm)$ of $\KK$-valued moulds such that
\be   \label{eqBirkhDec}
\bm^\est=\bp^\est=1_\kk, \qquad 
\bm-\idmoul \in \KK_-^{\underline\cN},\qquad 
\bp\in \KK_+^{\underline\cN},\qquad
\bm\times
T=
 \bp.
\ee
\bcf
Their values on an arbitrary word~$\un$ are determined by induction
on~$r(\un)$ by the formulae
$\bm^\est=\bp^\est=1_\kk$ and
\be   \label{eqinducdefbmbp}
\un\neq\est \imp
\bm^\un = - \pi_-(D^\un), \quad
\bp^\un = \pi_+(D^\un) \quad
\text{with} \ens
D^\un = \sum_{\un = \ua \, \ub,\; \ub\neq\est}
\ffff{\bm^\ua}\, T^\ub,
\ee
where $\pi_\pm \col \KK\to \KK_\pm$ are the projectors associated with the
decomposition $\KK=\KK_+\oplus \KK_-$.
\ecf

Moreover, $\bp$ and~$\bm$ are symmetral.
\end{proposition}

\begin{proof}\ 
\smallskip

$\bullet$ Uniqueness:
\fff{Suppose $(\bm,\bp)$ and $(\tbm,\tbp)$ satisfy~\eqref{eqBirkhDec}. }
We have $\bm^{-1}\times \bp=\tbm^{-1}\times \tbp$ so that
$\idmoul+\KK_-^{\UN}\ni\tbm\times \bm^{-1}=\tbp\times \bp^{-1}\in
\KK_+^{\UN}$. 
Therefore $\tbm\times \bm^{-1}=\tbp\times \bp^{-1}=\idmoul$, since $\KK_- \cap \KK_+=\{0\}$.

\smallskip

$\bullet$ Existence: 
\bcf
Let us define~$\bm$ and~$\bp$ by $\bm^\est=\bp^\est=1_\kk$ and~\eqref{eqinducdefbmbp}.
Setting $D\defeq \bm \times (T-\idmoul)$, we get
$\bm = \idmoul-\pi_-D$ and $\bp = \idmoul+\pi_+D$, whence
$\bp-\bm=D$, 
\ecf
\ie $\bp=\bm\times T$.

\smallskip

$\bullet$ Symmetrality:
%
%
Define the \emph{dimoulds} as the \fff{functions
  $\UN\times\UN \to\KK$}, and their multiplication as
\[
(M\times N)^{(\ua,\ub)} \defeq 
\sum_{(\ua,\ub) = (\ua^1,\ub^1) (\ua^2,\ub^2)} M^{(\ua^1,\ub^1)} N^{(\ua^2,\ub^2)},
\]
where the concatenation in $\UN\times\UN$ is defined by
$(\ua^1,\ub^1) (\ua^2,\ub^2) = (\ua^1 \, \ua^2, \ub^1 \, \ub^2)$.
A dimould is therefore the same as a linear form on the tensor square
of the shuffle algebra \fff{$\KUN$}.
\bcf
Dualizing~\eqref{eqshufflecoeff}, we define a map
$\De \col \KK^\UN \to \KK^{\UN\times\UN}$ by 
$(\Delta M)^{\ua,\ub} \defeq \sum\limits_{\un\in\UN} \shabn M^\un$
for any $(\ua,\ub) \in \UN\times\UN$.
According to~\cite{mouldSN}, $\De$ is a morphism of associative
algebras (thanks to the comptability between the comultiplication and
the shuffle product of $\KUN$) and, given $M\in\KK^\UN$,
\be\label{symm}
\text{\emph{$M$ is symmetral if and only if $M^\est=1$ and
$\De(M) = M \otimes M$,}}
\ee
with the notation 
$(M \otimes N)^{\ua,\ub} = M^\ua N^\ub$
for any $M,N\in\KK^\UN$.
\ecf

Let us define $A \defeq  \Delta \bm$ and $B \defeq  \Delta \bp$. Since $\bp =
\bm\times T$,
$A$ and~$B$ satisfy
\be  \label{eqstar}
B = A \times \Delta T,\qquad
A\in \idmoul+\KK_-^{\UN\times \UN}, \qquad
B\in \KK_+^{\UN\times\UN}. 
%
%
\ee
It is immediate to see that equation~\eqref{eqstar} has a unique (pair of dimoulds) solution, by the same argument as in the proof of the uniqueness part of Proposition~\ref{decomp}. Moreover the symmetrality of~$T$ implies that 
$\Delta T = T \otimes T$, and one checks easily that this implies that 
$(\bm\otimes \bm, \bp\otimes \bp)$ solves~\eqref{eqstar}
too. Therefore $\Delta \bm=\bm\otimes \bm$ and $\Delta \bp=\bp\otimes
\bp$, hence $\bm$ and~$\bp$ are symmetral by~\eqref{symm}.
\end{proof}


\subsection{Main results
}   \label{secMainResult}
%
\begin{theorem}\label{mainmain}
Let $T \in \KK^{\underline{\cN}}$ be defined by~\eqref{TouS}, and let
$(\bm,\bp)$ be its Birkhoff decomposition as stated in Proposition~\ref{decomp}.
Define the moulds $R, S \in \kk^{\underline{\cN}}$ by
\be  \label{eqdefSR}
R^{\underline{n}}=-r(\underline{n})(\mbox{residue of
}\bm^{\underline{n}}(\eps)),\qquad
S^{\underline{n}}=\mbox{ constant term of }
\bp^{\underline{n}}(\eps).
\ee
Then $(R,S)$ solves~\eqref{mouldeq}.
\end{theorem}

By ``constant term'' and ``residue'' of a Laurent series $\sum\limits_{n\in\Z} c_n\eps^n$ we mean respectively~$c_0$ and~$c_{-1}$.
%
%
\bcf
In view of Proposition~\ref{dyndyn}, Theorem~\ref{mainmain} entails

\begin{theorem}    \label{corolBD}
Define $\bm^\est(\eps) = \bp^\est(\eps) = 1_\kk$ and,
for nonempty $\un$, define
%
%
$\bm^\un(\eps) \in \eps\ii\kk[\eps\ii]$ and
$\bp^\un(\eps) \in \kk[[\eps]]$ by~\eqref{TouS} and~\eqref{eqinducdefbmbp},
and then $R^\un, S^\un \in \kk$ by~\eqref{eqdefSR}.
Then the mould expansions~\eqref{ansatz} provide a solution $(\Psi,N)$
to the normalization problem~\eqref{normal}.

\end{theorem}

\ecf 


The proof of Theorem~\ref{mainmain} will rely on

\begin{lemma}\label{lemaiii}
Let $S$ be as in~\eqref{eqdefSR} and $\ti R \defeq S\times I_\kk\times S^{-1}-\nabla_\ph S\times S^{-1}
\in \kk^{\UN} \subset \KK^{\UN}$.
Then
\begin{enumerate}[(i)]
\item \label{i}
$\nabla_\Phi \bm=-\ti R\times \bm$
\item\label{ii}
$\nabla_\Phi \bp=\bp\times I_\kk-\ti R\times \bp$
\item\label{iii}
$\ti R^\un=-(\eps\nabla_{1_\kk} \bm^\un(\eps))|_{\eps=\infty}$
\end{enumerate}
where $\na_{1_\kk}$ is the operator of $\KK^\UN$ defined by 
$M^{\underline{n}}(\eps) \defeq r(\underline{n})M^{\underline{n}}(\eps)$.
\end{lemma}

\begin{proof}[Proof of Lemma~\ref{lemaiii}]
Observe that $\na_\Phi = \na_\ph + \eps \na_{1_\kk}$ and that
$\na_{1_\kk}$, $\na_\ph$ and~$\na_\Phi$ are derivations of the
associative algebra~$\KK^\UN$.
Since $\bm\times T=\bp$ and $\nabla_\Phi T=T\times I_\kk$, we get
\[
\nabla_\Phi \bp=\bm\times T\times I_\kk+\nabla_\Phi \bm\times T=\bp\times I_\kk-\mathcal R\times \bp
\]
 with $\mathcal R \defeq -\nabla_\Phi \bm\times \bm^{-1}$.
So 
\be\label{plus}
\mathcal R = 
\bp\times I_\kk\times \bp^{-1}-\nabla_\Phi \bp\times \bp^{-1} \in \KK_+^{\underline{\cN}}
\ee
 since $\KK_+^{\underline{\cN}}$ is invariant by $\nabla_\Phi$.
On the other hand,
 \be\label{moins}
\mathcal R = 
-\na_\ph \bm\times \bm^{-1} -\eps \na_1 \bm\times \bm^{-1} =
P+\eps Q,\qquad P,Q\in \KK_-^{\underline{\cN}}.
\ee
 Since $\KK_+\cap \KK_-=\{0\}$, we deduce from
 (\ref{plus}--\ref{moins}) that
$\cR = (\eps Q)|_{\eps=\infty}$, \ie
 \[
 \mathcal R=-(\eps\nabla_1 \bm
 \times \bm^{-1}
 )|_{\eps=\infty}=-(\eps\nabla_1 \bm
 )|_{\eps=\infty}
 \]
  since $\bm^{-1}
\in \idmoul+\KK_-^{\underline\cN}$ so that $\bm^{-1}
|_{\eps=\infty}=\idmoul$.

Returning to~\eqref{plus}, since $\cR$ is constant in $\eps$, we get
\[
\mathcal R=(\bp\times I_\kk\times \bp^{-1}-\nabla_\Phi \bp\times
\bp^{-1})|_{\eps=0}=S\times I_\kk\times S^{-1}-\nabla_\ph S\times
S^{-1}=\ti R.
\]
 The three assertions \textit{(i)--(iii)} are proven.
\end{proof}

\begin{proof}[Proof of Theorem~\ref{mainmain}]
Lemma~\ref{lemaiii}\textit{(iii)} says that~$\ti R$ coincides with the
mould~$R$ defined by~\eqref{eqdefSR}, hence $\nabla_\ph S=S\times
I_\kk-R\times S$. 
The symmetrality of~$S$ follows from that of~$\bp$.
Therefore it is enough to prove that $\na_\ph R=0$.

We will show by induction on the length of $\underline{n}$ that
$\big[\ph(\underline{n})\neq 0\Rightarrow
\bm^{\un}(\eps)=0 \;\text{and}\; R^{\un}=0\big]$.
By definition $\ph(\est)=0$, nothing to prove.
Suppose $\ph(\un) \neq 0$.
By Lemma~\ref{lemaiii}\textit{(i)}, we have
\be\label{dim}
-(\ph(\un)+\eps r(\un))\bm^{\un}(\eps) = R^{\un} +
\sum\nolimits^* R^{\underline{a}}\,\bm^{\underline{b}}(\eps)
\ee
with $\sum^*$ representing summation over all non-trivial decompositions $\un=\ua\,\ub$,
because $\bm^\est=1_\kk$ and $R^\est=0$.
Since
$0\neq \ph(\un)=\ph(\underline{a})+\ph(\underline{b})$, at
least one of these two terms is different from~$0$, therefore the
induction hypothesis implies that
$R^{\underline{a}}\,\bm^{\underline{b}}(\eps)=0$, so the sum
in~\eqref{dim} vanishes. Moreover, since $\ph(\un)\neq 0$,
$\ph(\un)+\eps r(\un)$ is invertible in $\KK_+$, hence
$\KK_- \ni \bm^\un(\eps)=-\frac{R^\un}{\ph(\un)+\eps r(\un)}\in
\KK_+$ and therefore $\bm^{\underline{n}}(\eps)=0$ and $R^{\underline{n}}=0$.
\end{proof}

\bcf
\begin{remark}
Theorem~\ref{mainmain} could appear as a particular case of Theorem~5
of \cite{men2} were it not for the fact that the latter reference
deals with completed graded Lie algebras with
finite-dimensional components and their enveloping algebras,
a situation to which $\KK^\UN$ is not readily amenable.
\end{remark}
\ecf

\section{Proof of Theorem~\ref{main} and 
more}\label{mainmainq}
Given a self-adjoint operator~$H_0$ on a separable Hilbert space ~$\cH$ which is diagonal
in an orthonormal basis $\ebas = (e_k)_{k\in J\subseteq\N}$ 
with eigenvalues $E_0(k)$, one considers in \cite{Dyn} the space $\cL^\R \defeq \LRE[[\mu]]$ where 
$\LRE$ consists of all symmetric operators
whose domain is the dense subspace
$\Span_\C(\ebas)$ and which preserve $\Span_\C(\ebas)$. 
Since $\LRE$ is a Lie algebra over~$\R$ for the Lie bracket
$\qu{\cdot\,}\cdot \defeq \frac{1}{\I\hb}\times\text{commutator}$,
$\cL^\R$ is a complete filtered Lie algebra over~$\R$, filtered by order in~$\mu$.
In what follows, we denote commutators by $[\cdot\,,\cdot]$.

To decompose an arbitrary perturbation as a sum of eigenvectors of
$\ad_{H_0}\defeq \frac1{i\hbar}[H_0,\cdot]$, we notice that, for $B\in\cL^\R$ with matrix
$\big(\bet_{k,\ell}(\mu)\big)_{k,\ell\in J}$ on the basis $\ebas$
(with  $\bet_{k,\ell}(\mu) \in \C[[\mu]]$),
we can write
$
B = \sum\limits_{(k,\ell)\in J\times J} \bet_{k,\ell}(\mu)
\adir{\ell}{k} 
$.
The sum  might be infinite, but it is
well-defined because the action of $B$ in $\Span_\C(\ebas)$ is finitary. For the sake of simplicity, we suppose that $B$ is finite-band, which means that there exists $ D\in\N$ such that $\bet_{k,\ell}=0$ when $|k-\ell|>D$.

Since $\frac1{i\hbar}\big[H_0,\adir{\ell}{k} \big]=\frac{E_0(\ell)-E_0(k)}{i\hbar}\adir{\ell}{k}$,
we set
\be    \label{eqdefBlamQU}
\cN \defeq \big\{\, \tfrac{1}{\I\hb}(E_0(\ell)-E_0(k)) \mid 
(k,\ell) \in J\times J \,\big\} 
\ens\text{ and }\ens
B_\lam \defeq \sum_{ \substack{(k,\ell) \,\text{such that}\\[.5ex]
E_0(\ell)-E_0(k) = \I\hb\lam} } \bet_{k,\ell}(\mu) \adir{\ell}{k},
\ee
so that we have $B = \sum\limits_{\lam\in\cN} B_\lambda$ and
$\frac{1}{\I\hb}[H_0,B_\lam] = \lam B_\lam$
in the complex Lie algebra $\cL^\C \defeq \LCE[[\mu]]$,
where~$\LCE$ is defined like~$\LRE$ but without the symmetry requirement.
With these notations, we have the following result, more general than Theorem~\ref{main}:

\begin{theorem}\label{mainq}
\fff{Define, for $\ula\in\UN$,}
\[
N^{\underline{\lambda}}\defeq
-\,\text{residue of 
}
\bm^{\underline{\lambda}}(\eps), \qquad
S^{\underline{\lambda}} \defeq\,\mbox{constant term of }
\bp^{\underline{\lambda}}(\eps),
\] 
where $(\bm,\bp)$ is the Birkhoff decomposition of the
$\C((\eps))$-valued mould
\[
T^\ula(\eps) \defeq \frac1{(\lambda_1+\eps)(\lambda_1+\lambda_2+2\eps)\cdots(\lambda_1+\cdots+\lambda_{r(\ula)}+{r(\ula)\eps)})}
\]
\fff{inductively determined by~\eqref{eqinducdefbmbp} with 
$K_-=\eps^{-1}\C[\eps^{-1}]$ and $K_+= \C[[\eps]]$.}
Then the formulae
\begin{align}
\label{eqdefNihb}
N &\defeq \sum_{\underline\lambda\in\UN}
N^{\underline\lambda}\ 
\tfrac1{i\hbar}[B_{\lambda_1},\tfrac1{i\hbar}[B_{\lambda_2},\dots
\tfrac1{i\hbar}[B_{\lambda_{r(\underline\lambda)-1}},B_{\lambda_{r(\underline\lambda)}}]\dots]]\\[.5ex]
\label{eqdefPsiihb}
\Psi(\cdot) &\defeq
\sum_{\underline\lambda\in\UN}S^{\underline\lambda}\ 
\tfrac1{i\hbar}[B_{\lambda_1},\tfrac1{i\hbar}[B_{\lambda_2},\dots\tfrac1{i\hbar}[B_{\lambda_{r(\underline\lambda)}},\cdot\,]\dots]]
\end{align}
define respectively an element  of $\cL^\R$ and a unitary conjugation satisfying
\[
\Psi(H_0+B)=H_0+N \quad\text{and}\quad  [H_0,N]=0.
\]
Moreover, $\Psi(A) = C A\, C\ii$ with a unitary~$C$ given by the mould
expansion
\be  \label{eqdefCihb}
C \defeq \sum_{\underline\lambda\in\UN}
(\tfrac1{i\hbar})^{r(\ula)} S^{\ula} B_{\lam_1} B_{\lam_2} \cdots B_{\lam_{r(\ula)}}
\ee
\fff{(using the natural underlying structure of complete filtered associative algebra
of~$\cL^\C$).}
\end{theorem}


The proof of Theorem~\ref{mainq} requires two lemmas.

\begin{lemma} \label{lemSymalConj}
For any symmetral $S\in \C^\UN$, the formula~\eqref{eqdefPsiihb}
defines a Lie algebra automorphism~$\Psi$ of~$\cL^\C$ which is of the form
$\Psi(A) = C A\, C\ii$ with $C\in\cL^\C$ given by the mould
expansion~\eqref{eqdefCihb}.
\end{lemma}

\begin{proof}
  For arbitrary $D\in\cL^\C$, we use the notations $\gL_D \col A \mapsto D A$
  and $\gR_D \col A \mapsto A D$ for the left and right multiplication
  operators in the associative algebra~$\cL^\C$. Then we can
  rewrite~\eqref{eqdefPsiihb} as
\[
\Psi = \sum_{\ula\in\UN} (\tfrac{1}{\I\hb})^{\rula} \, S^\ula \,
(\gL_{B_{\lam_1}} - \gR_{B_{\lam_1}}) \cdots
(\gL_{B_{\lam_{\rula}}} - \gR_{B_{\lam_{\rula}}}).
\]
For each $\ula\in\UN$, since left and right multiplications commute,
we can expand
\[
(\gL_{B_{\lam_1}} - \gR_{B_{\lam_1}}) \cdots
(\gL_{B_{\lam_{\rula}}} - \gR_{B_{\lam_{\rula}}})
= \sum_{\ua,\ub} \shabl (-1)^{r(\ub)}
\gL_{B_{a_1}} \cdots \gL_{B_{a_{r(\ua)}}}
\fff{\gR_{B_{b_1}} \cdots \gR_{B_{b_{r(\ub)}}}}
\]
\fff{with the same shuffling coefficients as in~\eqref{eqshufflecoeff}}.
We thus get
\[
\Psi = \sum_{\ua,\ub,\ula} (-1)^{r(\ub)}(\tfrac{1}{\I\hb})^{\rula} \, \shabl S^\ula \,
\, \gL_{B_\ua} \,\gR_{B_{\ti\ub}},
\]
where $\ub \mapsto \ti\ub$ denotes word reversing. Symmetrality then
yields
\[
\Psi = \sum_{\ua,\ub} (-1)^{r(\ub)}(\tfrac{1}{\I\hb})^{r(\ua)+r(\ub)}
\, S^\ua \, S^\ub \, \gL_{B_\ua} \,\gR_{B_{\ti\ub}}
= \Big( \sum_\ua (\tfrac{1}{\I\hb})^{r(\ua)} \, S^\ua \, \gL_{B_\ua} \Big)
 \Big( \sum_\ub (-1)^{r(\ub)} (\tfrac{1}{\I\hb})^{r(\ub)} \, S^{\ti\ub} \, \gR_{B_\ub} \Big).
\]
We end up with $\Psi = \gL_C \,\gR_{\ti C}$, with~$C$ defined by the
mould expansion~\eqref{eqdefCihb}, and~$\ti C$ defined by the
analogous mould
expansion associated to~$\ti S$ defined by~\eqref{eqantipod}.
\fff{But $S\times\ti S = \ti S \times S= \idmoul$, by symmetrality of~$S$,
and this clearly entails $C\ti C = \ti C C = \ID_\cH$.}
%
%
\end{proof}


\begin{lemma} \label{lemSymalConj}
For any $N\in \C^\UN$ such that the complex conjugate of $N^{\lam_1,\dots,\lam_r}$ is
$N^{-\lam_1,\dots,-\lam_r}$, the mould expansion $N \in \cL^\C$
defined by~\eqref{eqdefNihb} is in~$\cL^\R$.

For any symmetral $S\in \C^\UN$ such that the complex conjugate of $S^{\lam_1,\dots,\lam_r}$ is
$S^{-\lam_1,\dots,-\lam_r}$, the mould expansion $C \in \cL^\C$
defined by~\eqref{eqdefCihb} is unitary.
\end{lemma}

\begin{proof}
Observe that the adjoint of the operator~$B_\lam$ is~$B_{-\lam}$ for
every $\lam\in\cN$.
Since taking the adjoint is a real Lie algebra automorphism of~$\cL^\C$,
this yields that the mould expansion~$N$ is a symmetric operator.

In the case of~$C$, we find that the adjoint is given by the mould
expansion associated to~$\ti S$ defined by~\eqref{eqantipod}, which
is~$C\ii$ as already mentioned.
\end{proof}


\begin{proof}[Proof of Theorem~\ref{mainq}]
Apply \fff{Theorem~\ref{corolBD} to the normalization problem in $\cL^\C = \LCE[[\mu]]$ viewed as
  Lie algebra over $\kk=\C$ with Lie bracket $\qu{\cdot\,}\cdot$,
  filtered by order in~$\mu$, and with $\ph(\lam) \equiv\lam$
in~\eqref{TouS}}. Observe that, since $\cN \subset \I\,\R$, the
complex conjugate of $T^{\lam_1,\dots,\lam_r}$ is
$T^{-\lam_1,\dots,-\lam_r}$;
it easy to see that this property is inherited by~$\bm$ and~$\bp$, and
hence by the constant moulds~$N$ and~$S$.
\end{proof}


\bcf

\begin{remark}
  Using the $\C$-valued mould $G=\log S$ defined in appendix, we see
  that $C = \ex^{\frac{1}{i\hbar}W}$ with $W =  
\sum\limits_{\ula\in\UN\setminus\{\est\}} 
\frac{1}{r(\ula)} G^\ula \,
\tfrac1{i\hbar}[B_{\lambda_1},\tfrac1{i\hbar}[B_{\lambda_2},\dots
\tfrac1{i\hbar}[B_{\lambda_{r(\underline\lambda)-1}},B_{\lambda_{r(\underline\lambda)}}]\dots]
\in \cL^\R$.
\end{remark}


\begin{proof}[Proof of Theorem~\ref{main}]
Take~$V$ in Section~\ref{intro} as $\mu B \in
\fff{\cL^\R = \LRE[[\mu]]}$, 
and identify the homogeneous terms in $\mu$ and in $V$ (see the
Addendum of Theorem A in \cite{Dyn} for a more precise
statement). 
\end{proof}

\ecf


\section{Extensions}\label{sect3}

In \cite{Dyn}, four other examples of complete filtered algebras are
considered, corresponding to four dynamical situations:
Poincar\'e-Dulac normal forms, Birkhoff normal forms, multiphase
averaging and the semiclassical approximation of the situation of the
present article. In all these examples, as in Section~\ref{mainmainq},
the results are derived exclusively out of a mould equation of the
form~\eqref{mouldeq}. Therefore statements similar to
theorem~\ref{mainq} can be established.

%
More quantitative results are proven in \cite{Norm} in the situation
of an equation of the form~\eqref{normal} stated on Banach scales of
Lie algebras: precise estimates (in convenient norms) are given when
mould expansions are truncated. They also rely exclusively on mould
equations and so can be rephrased using Birkhoff decompositions.

Precise formulations for all these cases are left to the interested reader.

\bcf
\appendix
\section{}

\subsection{Mould exponential and alternality}  \label{secmldexpalt}
%
Let $\kk$ be a ring and~$\cN$ a nonempty set, and consider the set of moulds
$\kk^\UN$ as in Section~\ref{secMoulds}.
We define a decreasing filtration by declaring that, for $m\ge0$, a
mould~$M$ is of order $\geq m$ if $M^\un=0$ whenever $r(\un)<m$; this
is easily seen to be compatible with mould multiplication, and in fact
$\kk^\UN$ is a complete filtered associative algebra.
We can thus define the mutually inverse exponential and logarithm maps by the usual series
\[
M^\est = 0 \imp
\ex^{M} \defeq \idmoul + 
\sum_{k\ge1} \tfrac{1}{k!} (M)^{\times k},
\quad
\log(\idmoul+M) \defeq \sum_{k\ge1} \tfrac{(-1)^{k-1}}{k} (M)^{\times k},
\]
which are formally summable (only finitely many terms contribute to
the evaluation of $\ex^{M}$ or $\log(\idmoul+M)$ on a given word).

A mould~$M$ is said to be ``alternal'' if $M^\est=0$ and
$\sum\limits_{\un\in\UN} \shabn M^\un = 0$
for any nonempty words~$\ua$ and~$\ub$.
Equivalently, using the map~$\De$ mentioned in the paragraph containing~\eqref{symm},
$M$~is alternal if and only if $\De M = M\otimes\idmoul +
\idmoul\otimes M$.
Since~$\De$ is a morphism of associative algebras,
alternal moulds form a Lie subalgebra of $\LIE(\kk^\UN)$
(the space $\kk^\UN$ viewed as a Lie algebra for which bracketing is defined by commutators).

The exponential map $M\mapsto \ex^M$ induces a bijection between
alternal moulds and symmetral moulds (use~$\De$ and~\eqref{symm}).

Notice that, when identifying moulds with linear forms on the shuffle
algebra, alternal moulds are identified with infinitesimal characters:
\be
M^{\ua\shuffle \ub} = \eta(\ua) M^\ub + \eta(\ub) M^\ua 
\quad\text{for all $\ua,\ub\in\UN$,}
\ee
where we denote by~$\eta$ the counit.

\subsection{Proof of Proposition~\ref{dyndyn}}

Suppose that, in the situation described at the beginning of
Section~\ref{sect2}, we have a solution $(R,S)$ to the mould
equation~\eqref{mouldeq}.
We must prove that the mould expansions~\eqref{ansatz} define a
solution $(N,\Psi)$ to~\eqref{normal}.

Let us introduce the notations 
$\Bun \defeq [B_{n_1},[\ldots[B_{n_{r-1}},B_{n_r}]\ldots]]$ and
$\gB_\un \defeq \ad_{B_{n_1}} \cdots \ad_{B_{n_r}}$ 
for an arbitrary word $\un = n_1\cdots n_r$, with the conventions $\Best =0$ and $\gB_\est=\ID$.
Because $\cL$ is a complete filtered Lie algebra and each $B_{n_i} \in
\cL\ugeq1$, it is easily checked that one can define two linear maps
$\gL \col \kk^\UN \to \cL$ and $\gE \col \kk^\UN \to \End_\kk(\cL)$
($\kk$-linear operators)
by the formulae
\be   \label{eqdefgLgE}
\gL(M) \defeq 
\sum_{\un\in\UN\setminus\{\est\}} \tfrac{1}{r(\un)} M^\un \Bun,
\qquad
\gE(M) \defeq 
\sum_{\un\in\UN} M^\un \gB_\un.
\ee
In particular, $N = \gL(R)$ and $\Psi = \gE(S)$ are well-defined.

As already mentioned in the paragraph containing~\eqref{eqgenLeibn},
$\Psi$ is a Lie algebra automorphism because~$S$ is symmetral.
By induction on~$r(\un)$, we deduce from~\eqref{eqBneigenv} that
$[X_0,\Bun] = \ph(\un)\Bun$,
whence 
\be   \label{eqadXznaph}
[X_0,\gL(M)] = \gL(\na_\ph M)
\ens\text{for any mould~$M$.}
\ee
In particular, $\na_\ph R=0$ entails $[X_0,N] = 0$,
and we are just left with the verification of the first relation in~\eqref{normal}.
This will be obtained by means of the two identities
\begin{gather}
\label{eqPsiB}
\Psi(B) = \gL(S\times I_\kk\times S\ii),
\\[1ex]
\label{eqPsiXz}
\Psi(X_0)-X_0 = - \gL(\na_\ph S \times S\ii),
\end{gather}
the sum of which will yield the desired result, namely
$\Psi(X_0+B)-X_0=N$,
in view of the relation
$S\times I_\kk\times S\ii - \na_\ph S \times S\ii = R$ granted by the
mould equation.

Before proving~\eqref{eqPsiB} and~\eqref{eqPsiXz}, we show that $\Psi = \ex^{\ad_W}$
with~$W$ in the range of~$\gL$.
Let $G \defeq \log S$. As explained in Section~\ref{secmldexpalt},
this is an alternal mould.
Since $\gB_{\ua\,\ub} = \gB_\ua \gB_\ub$, the map~$\gE$ is clearly a
morphism of filtered associative algebras, hence
$\Psi = \gE(\ex^G) = \ex^{\gE(G)}$.
Let $\gBun \defeq
[\ad_{B_{n_1}},[\ldots[\ad_{B_{n_{r-1}}},\ad_{B_{n_r}}]\ldots]]$ 
for an arbitrary word $\un = n_1\cdots n_r$.
The alternality of~$G$ entails
\be   \label{eqrelDSW}
\gE(G) = \sum_{\un\in\UN} G^\un \gB_\un =
\sum_{\un\in\UN\setminus\{\est\}} \tfrac{1}{r(\un)} G^\un \gBun
= \ad_{\gL(G)}.
\ee
Indeed, the middle equality in~\eqref{eqrelDSW} is obtained for any alternal mould from the
identity
\[
\gBun = \sum_{(\ua,\ub)\in\UN\times\UN}
(-1)^{r(\ub)} r(\ua) \shabn \ffff{\gB_{\ua \, \ti\ub}}
\quad\text{for all $\un\in\UN$}
\]
(where we denote by $\ub\mapsto\ti\ub$ order reversal), which results
from a classical computation (related to the Dynkin-Specht-Wever
idempotent --- see \cite{vonW} or \cite{Dyn}), and the last equality
in~\eqref{eqrelDSW} follows from the obvious relation $\gBun = \ad_{\Bun}$.
Therefore, $\Psi = \ex^{\ad_W}$ with $W = \gL(G)$.


\begin{proof}[Proof of~\eqref{eqPsiB}]
An identity similar to the middle equality in~\eqref{eqrelDSW}, but at the level of~$\cL$ and
its universal enveloping algebra, implies that \emph{the restriction of~$\gL$ to
  alternal moulds is a morphism of Lie algebras}.
It follows that $\ad_W\big(\gL(M)\big) = \gL(\ad_G M)$ for any
alternal~$M$
(denoting by~$\ad$ the adjoint representations of~$\cL$ and $\LIE(\kk^\UN)$), hence
\[
\ex^{\ad_W} \gL(M) = \gL(\ex^{\ad_G} M) = 
\gL(S\times M\times S\ii)
\]
(we have used Hadamard's lemma in $\kk^\UN$ for the last equality:
$\ex^{\ad_G} M = \ex^G\times M\times \ex^{-G}$).
Since~$I_\kk$ is an alternal mould satisfying $\gL(I_\kk) = B$, we get
as a particular case 
$\Psi(B) = \ex^{\ad_W} \gL(I_\kk) = \gL(S\times I_\kk\times S\ii)$.
\end{proof}


\begin{proof}[Proof of~\eqref{eqPsiXz}]
By~\eqref{eqadXznaph}, $\ad_W X_0 = - \gL(\na_\ph G)$. 
Using again the morphism of Lie algebras induced by~$\gL$, we derive
$\ad_W^k X_0 = - \ad_W^{k-1} \gL(\na_\ph G) = - \gL(\ad_G^{k-1}\na\ph
G)$ for all $k\ge1$,
whence 
\[
\Psi(X_0) - X_0 = - \gL(M), \qquad
M \defeq \sum_{k\ge1} \tfrac{1}{k!} \ad_G^{k-1}\na\ph G.
\] 
A classical computation\footnote{%
Check that $\sum
\frac{1}{k!}(U-V)^{k-1} = \sum
\frac{1}{(p+q+1)!} U^p V^q \ex^{-V}$
in $\Q[[U,V]]$ 
\eg by multiplying both sides by $(U-V)\ex^V$,
substitute for~$U$ and~$V$ the operators of left and right
multiplication by~$G$ in $\kk^\UN$ \emph{which commute}, 
apply the resulting operator to~$\na_\ph G$, and remember
that~$\na_\ph$ is a derivation.}
yields $M = \na_\ph(\ex^G) \times \ex^{-G}$, whence
$\Psi(X_0) - X_0 = - \gL(\na_\ph S \times S\ii)$,
as desired.
\end{proof}


\begin{remark}
There is another proof of Proposition~\ref{dyndyn}, which consists in defining on
the universal enveloping algebra $U(\cL)$ of~$\cL$ a decreasing filtration of associative
algebra which is separated and complete, so as to be able to define a
morphism of filtered associative algebras $\gU\col\kk^\UN \to U(\cL)$
analogous to~$\gE$
(this extra work can be dispensed with in the case of the Lie
algebra~$\cL^\C$ of Section~\ref{mainmainq}, since it has a natural
structure of complete filtered associative algebra).
One then checks that the restrictions of~$\gU$ and~$\gL$ to alternal
moulds coincide, and that the normalization problem is
solved by $N\defeq\gU(R)$ and the conjugation automorphism $\Psi \col
A \mapsto C A C\ii$ where $C\defeq\gU(S)$
(because $\gU(\na_\ph S) = [X_0,C]$, $\gU(S\times I_\kk) = C B$ and
$\gU(R\times S) = NC$).
\end{remark}

\ecf


\small

\smallskip

\noindent \textbf{Acknowledgments}: 
This work has been partially carried out thanks to the support of the A*MIDEX project (n$^o$ ANR-11-IDEX-0001-02) funded by the ``Investissements d'Avenir" French Government program, managed by the French 
National Research Agency (ANR). T.P. thanks also the Dipartimento di
Matematica, Sapienza Universit\`a di Roma, for its kind hospitality
during the elaboration of this work.
D.S. thanks Fibonacci Laboratory (CNRS UMI~3483), the Centro Di
Ricerca Matematica Ennio De Giorgi and the Scuola Normale Superiore di
Pisa for their kind hospitality. This work has received funding from the
French National Research Agency under the reference ANR-12-BS01-0017.


%
\scriptsize
\end{document}